\documentclass[11pt]{amsart}
\usepackage{hyperref}
\usepackage{amssymb,amsfonts}
\usepackage{hyperref}
\usepackage{amssymb,textcomp}

\usepackage{enumerate}

\usepackage[utf8]{inputenc}
\usepackage[T1]{fontenc}

\usepackage[mathscr]{eucal}

\usepackage[a4paper, twoside=false, vmargin={2cm,3cm}, includehead]{geometry}

\newtheorem{lemma}{Lemma}[section]
\newtheorem{theorem}[lemma]{Theorem}
\newtheorem*{theorem*}{Theorem}
\newtheorem{corollary}[lemma]{Corollary}

\newtheorem{proposition}[lemma]{Proposition}
\newtheorem*{conjecture*}{Conjecture}
\newtheorem*{proposition*}{Proposition}

\newtheorem{conjecture}{Conjecture}

\theoremstyle{definition}
\newtheorem*{claim*}{Claim}

\newtheorem{definition}[lemma]{Definition}

\newtheorem*{remark}{Remark}
\newtheorem*{remarks}{Remarks}

\newcommand{\lE}{\mathbb{E}^{\log}}

\newcommand{\C}{{\mathbb C}}
\newcommand{\D}{{\mathbb D}}
\newcommand{\E}{{\mathbb E}}

\newcommand{\N}{{\mathbb N}}
\renewcommand{\P}{{\mathbb P}}

\newcommand{\R}{{\mathbb R}}
\newcommand{\T}{{\mathbb T}}
\newcommand{\Z}{{\mathbb Z}}

\newcommand{\CX}{{\mathcal X}}

\newcommand{\CK}{{\mathcal K}_{\text{rat}}}

\newcommand{\ux}{\underline x}

\newcommand{\bM}{{\mathbf{M}}}
\newcommand{\bN}{{\mathbf N}}

\newcommand{\bmu}{{\boldsymbol{\mu}}}

\newcommand{\wt}{\widetilde}
\newcommand{\e}{\mathrm{e}}

\DeclareMathOperator{\id}{id}

\DeclareMathOperator{\reel}{Re}
\renewcommand{\Re}{\reel}

\begin{document}

	\title[Furstenberg systems of  multiplicative functions and applications]{Furstenberg systems of bounded multiplicative functions and applications}


\author{Nikos Frantzikinakis}
\address[Nikos Frantzikinakis]{University of Crete, Department of mathematics, Voutes University Campus, Heraklion 71003, Greece} \email{frantzikinakis@gmail.com}
\author{Bernard Host}
\address[Bernard Host]{
	Universit\'e Paris-Est Marne-la-Vall\'ee, Laboratoire d'analyse et
	de math\'ematiques appliqu\'ees, UMR CNRS 8050, 5 Bd Descartes,
	77454 Marne la Vall\'ee Cedex, France }
\email{bernard.host@u-pem.fr}

\begin{abstract}
	We prove a structural result for measure preserving systems naturally associated with any finite collection of multiplicative functions that take values on the complex unit disc. We show that these systems have no irrational spectrum and  their building blocks are Bernoulli systems and infinite-step nilsystems.
	One consequence of our structural result is that strongly aperiodic multiplicative functions
	satisfy the logarithmically averaged variant of the disjointness conjecture of Sarnak  for a wide class of zero entropy topological dynamical systems, which includes all uniquely ergodic ones. We deduce that aperiodic
	multiplicative functions with values plus or minus one have super-linear  block growth.
	Another consequence of our structural result is
	that  products of shifts of arbitrary  multiplicative functions with values on the unit disc do not correlate with  any  totally ergodic deterministic sequence of zero mean.
	Our methodology is based primarily on techniques developed in a previous article of the authors where analogous results were proved for the M\"obius and the Liouville function. A  new ingredient needed is a result obtained recently by  Tao and  Ter\"av\"ainen related to the  odd order cases of the Chowla conjecture.
\end{abstract}



\maketitle

\section{Introduction and main results}
\subsection{Main results} A function $f\colon \mathbb{N}\to \mathbb{C}$ is called \emph{multiplicative} if
$$
f(mn)=f(m)f(n) \ \text{ whenever } \  (m,n)=1.
$$
 Perhaps the most   well-known example of a bounded multiplicative function is the
  M\"obius function,   which
  is defined to be $0$ on integers divisible by a square,
   $-1$ on square-free  integers with an odd number of prime factors, and $1$ elsewhere.
Its  non-zero values are expected to fluctuate between $-1$ and $1$ in  a random way and many famous conjectures have been formulated based on this belief. One example that has received a lot of attention in recent years is the M\"obius disjointness conjecture of  Sarnak~\cite{Sa, Sa12}. It   asserts that  the M\"obius function  does not correlate with any bounded deterministic sequence, meaning, any sequence that is produced by   a continuous function evaluated along the orbit of a point in a zero entropy topological dynamical system.

 In  \cite{FH18} we
verified the logarithmically averaged variant of the conjecture  of Sarnak for a wide class of deterministic sequences. Our approach was to study  measure preserving systems (which we call Furstenberg systems) naturally associated with the M\"obius function; in particular, we studied structural properties that
  allow to deduce disjointness  from  a wide class of zero entropy systems. Various interesting results, establishing non-correlation of deterministic sequences with the M\"obius function and products of its shifts,  are natural consequences
  of these disjointness results.

The main purpose of this article is to extend the approach  from \cite{FH18} in order to cover    multiplicative functions
that take   values on the complex unit disc $\mathbb{U}:=\{z\in \C\colon |z|\leq 1\}$. 
Our first main result concerns a class of multiplicative functions that are expected to satisfy similar disjointness properties
 as   the M\"obius function. These are the
 strongly aperiodic multiplicative functions
 (see Definition~\ref{D:StronglyAperiodic}), and we
  verify that they do not correlate  with a wide class of deterministic sequences.
  Throughout the paper, we denote by $(Y,R)$ a \emph{topological dynamical system}, meaning, a compact metric space $Y$ together with a continuous homeomorphism $R\colon Y\to Y$.
\begin{theorem}\label{T:Sarnak1}
	Let $f\colon\N\to \mathbb{U}$ be a strongly  aperiodic multiplicative function.
	Let $(Y,R)$ be a topological  dynamical system
	 with zero topological entropy and at most countably many  ergodic invariant measures. Then for every  $y\in Y$  and  every $g\in C(Y)$  we have
	$$
	\lim_{N\to\infty} \frac 1{\log N}\sum_{n=1}^N\frac{g(R^ny)\, f(n)}{n} =0.
	$$
		Furthermore, the convergence is uniform in $y\in Y$.
\end{theorem}
\begin{remarks}
$\bullet$
Using rotations on finite cyclic groups,  one deduces that non-correlation (using logarithmic averages) of $f$ with all periodic sequences (which implies strong aperiodicity in the real valued case) is a necessary assumption    for the conclusion to hold.

$\bullet$ We believe that the countability assumption on the number of ergodic invariant measures  of $(Y,R)$ can be dropped.
In the case where $f$ is  the M\"obius function, this is equivalent to the logarithmically averaged variant of the
Sarnak conjecture.
\end{remarks}
An interesting consequence of the previous result  is a statement about the block complexity
 of  multiplicative functions $f\colon \N\to \mathbb{U}$ that have finite range.  In the next statement we denote by $P_f(n)$
     the number of   patterns of size $n$  that  are taken by consecutive values of  $f$ (see Section~\ref{SS:Superlinear} for a more formal definition).
\begin{theorem}\label{T:superlinear}
	If the  multiplicative function  $f\colon\N\to \mathbb{U}$ has finite range,  is  strongly  aperiodic,  and does not converge to zero in logarithmic density, then $\lim_{n\to\infty}\frac{P_f(n)}{n}=\infty$.
\end{theorem}
\begin{remarks}
	$\bullet$ 	
	 In fact, we establish a stronger statement, if $a\colon \N\to \C$ has finite range and linear block growth, then  $\lim_{N\to\infty}\frac{1}{\log{N}}\sum_{n=1}^N
	 \frac{a(n) \, f(n)}{n}=0$ for every strongly aperiodic  multiplicative function $f\colon \N\to\mathbb{U}$. Thus,  even if we modify the values of $f$ on a set of logarithmic density $0$, using values taken from a finite set of complex numbers,  the new sequence still has super-linear block growth.

$\bullet$ The assumptions are satisfied if
$f$ takes only the values  $\pm 1$ and   is aperiodic, meaning, it does not correlate with any periodic sequence.
Previously,  it was not even known
that for such multiplicative functions we have  $\lim_{n\to \infty}(P_f(n)-n)=\infty$. On the other hand, a conjecture of  Elliott~\cite{El90,El94} predicts if $f\colon \N\to \{-1,1\}$ is aperiodic, then  $P_f(n)=2^n$ for every $n\in\N$, and if $f\colon \N\to\mathbb{U}$ has finite range, is strongly aperiodic, and does not converge to zero in logarithmic density,  then  $P_f(n)$ grows exponentially.
\end{remarks}
Henceforth, whenever needed, we  assume that a multiplicative function $f\colon \N\to \mathbb{U}$ is extended to the integers in an arbitrary way.

In  the next result if  $(Y,R)$ is a topological dynamical system, we say that  a point  $y\in Y$ is generic for logarithmic averages for a Borel probability measure $\nu$ on $Y$  if for every $g\in C(Y)$ we have
$
\lim_{N\to\infty}	\frac{1}{\log{N}}\sum_{n=1}^N \frac{g(R^ny)}{n}=\int g\, d\nu.
$
Our methods also allow us to prove  non-correlation
between products of shifts of arbitrary multiplicative functions with values on the unit disc and totally ergodic deterministic sequences of zero mean.
\begin{theorem}\label{T:Sarnak2}
	Let  $f_1,\ldots,f_\ell\colon\N\to \mathbb{U}$ be  multiplicative functions.
	Let $(Y,R)$ be a topological  dynamical system
	 and let
	$y\in Y$ be   generic for logarithmic averages for a  measure $\nu$
	with zero  entropy and at most countably many  ergodic components,  all of which are totally ergodic.
	Then  for every $g\in C(Y)$  that is orthogonal in $L^2(\nu)$ to all $R$-invariant functions
 we have
	\begin{equation}\label{E:weightedElliott}
	\lim_{N\to\infty}		\frac 1{\log N}\sum_{n=1}^N\frac{g(R^ny)\, \prod_{j=1}^\ell f_j(n+h_j)}{n}=0
	\end{equation}
	for all   $h_1,\dots,h_\ell \in\Z$.
\end{theorem}

\begin{remarks}	
$\bullet$	The unweighted version of \eqref{E:weightedElliott}  (take $g:=1$)  is expected to hold  if
	the shifts are distinct and at least one of the multiplicative functions is strongly aperiodic. This is the logarithmically averaged variant of  a  conjecture of Elliott~\cite{El90,El94} (see \cite[Theorem~B.1]{MRT15} for a corrected version and the  need to assume strong aperiodicity).

$\bullet$ 	If  $(Y,R)$ has zero topological entropy and is uniquely and totally ergodic, then it is easy to deduce from Theorem~\ref{T:Sarnak2} that
\eqref{E:weightedElliott} holds for all $g\in C(Y)$ and $y\in Y$ such that
$\lim_{N\to\infty}\frac 1{\log N}\sum_{n=1}^N\frac{g(R^ny)}{n}=0$.

$\bullet$  If $(Y,\nu,R)$ is totally ergodic, then using an approximation argument one can also conclude that \eqref{E:weightedElliott} holds for every $g\colon Y\to \C$ that is Riemann integrable with respect to the measure $\nu$ and $\int g\, d\nu=0$.\footnote{We say that a function $g\colon Y\to \C$ is
	Riemann integrable with respect to the measure $\nu$ if  for every $\varepsilon>0$ there exist $g_1,g_2\in C(Y)$ such that
	$g_1(y)\leq g(y)\leq g_2(y)$ for every $y\in Y$ and $\int (g_2-g_1)\, d\nu\leq \varepsilon$.}
	\end{remarks}
Theorem~\ref{T:Sarnak2}  is new even in the very special case where $R$ is given by  an irrational rotation on $\T$ and $g(t):=\e^{2\pi i t}$, $t\in \T$. In this case we have $g(R^n0)=\e^{2\pi i n\alpha}$, $n\in\N$,  for some irrational $\alpha$, and we get the following result as a consequence:
\begin{corollary}
	\label{C:main-2}
		Let  $f_1,\ldots,f_\ell\colon\N\to \mathbb{U}$ be  multiplicative functions and let
	 $\alpha\in \R$ be irrational. Then
	\begin{equation} \label{E:irrational}
	\lim_{N\to\infty}		\frac 1{\log N}\sum_{n=1}^N\frac{\e^{2\pi  i n \alpha}\, \prod_{j=1}^\ell f_j(n+h_j)}{n}=0
	\end{equation}
	for all   $h_1,\dots,h_\ell \in\Z$. 	
\end{corollary}
\begin{remarks} $\bullet$ For $\ell=1$ the result is
	the logarithmically averaged variant of a classical result of  Daboussi~\cite{D74, DD74, DD82}. But even for $\ell=2$ the result is new.
	
	$\bullet$	 More generally, if we  apply Theorem~\ref{T:Sarnak2} for $R$  given by appropriate totally ergodic affine transformations on a torus with the Haar measure (as in \cite[Section~3.3]{Fu77}), we get that
	\eqref{E:irrational} holds with $(\e^{2\pi i n\alpha})_{n\in\N}$ replaced by any sequence of the form  $(\e^{2\pi i P(n)})_{n\in \N}$, where $P\in \R[t]$ has at least one non-constant coefficient   irrational.
	Moreover,  one could use as weights zero mean sequences arising from more general totally ergodic nilsystems, giving rise to generalized polynomial sequences. One such example is  the sequence $(e^{2\pi i [n\alpha]n\beta})_{n\in\N}$, where $\alpha,\beta \in \R$ are rationally independent. In order to establish this variant, one has to use
  Theorem~\ref{T:Sarnak2} (see the last remark following this result) for a totally ergodic nilsystem $(Y,R)$ defined on the Heisenberg nilmanifold  and an appropriate  Riemann integrable function  $g$ with respect to the Haar measure on $Y$ with zero integral (see~\cite[Section~0.14]{BL07} for details).
\end{remarks}

A key step in the proof of the previous results
is   a
structural result for measure preserving systems naturally associated with any collection of multiplicative functions that take values on the complex unit disc. We call such systems \emph{Furstenberg systems}, and they  are defined as follows: For convenience, let $f$ be a multiplicative function that takes values on a finite subset $A$ of $\mathbb{U}$ and admits correlations for logarithmic averages  on a sequence of intervals $\bN=([N_k])_{k\in\N}$ with $N_k\to\infty$ (see Definition~\ref{D:correlations}).
Then the   Furstenberg system  associated with $f$ and $\bN$ is defined on the  sequence space $X:=A^\Z$ with the shift transformation,
by a measure that assigns to each  cylinder set $\{x\in X\colon x(j)=a_j,j=-m,\ldots,m\}$ value equal to the logarithmic density, taken along the sequence $\bN$,  of the set $\{n\in \N\colon f(n+j)=a_j,j=-m,\ldots,m\}$, where
 $a_{-m},\ldots, a_m\in A$ and  $m\in \N$.
 Similarly, one defines Furstenberg systems associated with any finite collection of multiplicative functions $f_1,\ldots, f_\ell\colon \N\to \mathbb{U}$ and a sequence of intervals $\bN$ on which $f_1,\ldots, f_\ell$  admit correlations for logarithmic averages;  we call these measure preserving systems \emph{joint Furstenberg systems}.
 The precise constructions  are given in
 Section~\ref{SS:Furst} and are motivated by analogous constructions made by Furstenberg  in \cite{Fu77,Fu} in order to restate  Szemer\'edi's theorem on arithmetic progressions in ergodic terms.
  We prove the following structural result for joint Furstenberg systems of multiplicative functions:
\begin{theorem}
	\label{T:structure}
	A	joint Furstenberg system
	of  the multiplicative functions $f_1,\ldots, f_\ell \colon \N\to \mathbb{U}$  is a factor of a system  that
	\begin{enumerate}
		\item
		has no irrational spectrum;
		
		\item  has ergodic components
		isomorphic to   direct products of infinite-step nilsystems and  Bernoulli systems.	
	\end{enumerate}
\end{theorem}
\begin{remarks}
	$\bullet$ We refer the reader to
	 Section~2 and
	 Appendix~A of \cite{FH18} for the definition of the ergodic notions used in the previous statement.
	
	 $\bullet$ The product decomposition depends on the ergodic component, in particular,  the infinite-step nilsystems depend on  the  ergodic component.

	\end{remarks}
See Section~\ref{SS:conj} for more refined conjectural statements regarding the structure of joint
Furstenberg systems of  multiplicative functions with values on the unit disc.

\subsection{Proof strategy} Our general strategy in the proofs of Theorems~\ref{T:Sarnak1}-\ref{T:Sarnak2} is similar to the one used in \cite{FH18} to cover the case of the M\"obius and the Liouville  functions, but there are also some serious additional difficulties that we have to overcome.
Our main focus  is to prove the structural result
stated in Theorem~\ref{T:structure}; then
Theorems~\ref{T:Sarnak1}-\ref{T:Sarnak2} are consequences of this result and   the deduction is carried out using  joining arguments  in Section~\ref{S:123}
(Theorem~\ref{T:Sarnak1} also uses additional  number theory input provided by
Theorem~\ref{T:Tao}).
The first step in the proof of Theorem~\ref{T:structure} is to  apply  the identity of Theorem~\ref{T:Tao1}
which
allows to express an arbitrary joint correlation 
of multiplicative functions
as a weighted average of their
dilated joint correlations taken over all  prime dilates   (this step necessitates the use of logarithmic averages).
This leads, via the correspondence principle of Furstenberg (see Proposition~\ref{P:correspondence}), to certain ergodic  identities  that any  joint Furstenberg system $(X,\mu,T)$ of these
multiplicative functions satisfies.

 The next goal is to utilize the ergodic  identities in order to prove the structural properties of Theorem~\ref{T:structure}. Unfortunately, the presence of
some unwanted weights, which appear  because the multiplicative functions
are not constant on primes, creates serious problems that do not allow us to continue as in \cite{FH18}, especially when the multiplicative functions take infinitely many distinct values on primes.
The way to overcome this obstacle is to first utilize a recent result of  Tao and  Ter\"av\"ainen, which proves that joint correlations of multiplicative functions vanish if
the product of the multiplicative functions is far from being periodic. This result enables us to obtain a variant
of the identity  in Theorem~\ref{T:Tao1}, which  has the additional  property that all the weights are equal to $1$
(see Corollary~\ref{C:constantone}).
As a consequence, we get  an ergodic identity, stated  in Theorem~\ref{T:Tao2},  which  allows to show that  the system $(X,\mu,T)$   is a factor of some system of arithmetic progressions with steps given by all
 primes  in an appropriate congruence class
 (see Definition~\ref{D:tilde}). Finally, this system
can be easily linked to a system of arithmetic progressions with prime steps (see Lemma~\ref{L:abscont}). The structure of these systems was  studied in \cite{FH18} and they were  shown to satisfy the structural properties of Theorem~\ref{T:structure}. Combining the above facts we get a proof of Theorem~\ref{T:structure}.

A simpler and  more elementary way to
link Furstenberg systems of multiplicative functions to  systems of arithmetic progressions with primes steps is explained in Section~\ref{SS:alternative}; but this simpler approach only
works   if the range of the multiplicative functions on the primes is a subset of the  unit interval or a finite subset of the complex unit disc.

\subsection{Further remarks and conjectures} \label{SS:conj}
The structural result of Theorem~\ref{T:structure} is not expected to be optimal and we give below some more refined
conjectural structural statements.

 In what follows, unless explicitly specified, a Bernoulli system is allowed to be the trivial one point system. Moreover, an ergodic procyclic system (often referred to as an odometer) is an ergodic inverse limit of periodic systems,
	or equivalently, an ergodic system $(X,\mu,T)$ for which   the rational eigenfunctions span a dense subspace of $L^2(\mu)$.
	\begin{conjecture} If the multiplicative functions  $f_1,\ldots, f_\ell$ take values in $[-1,1]$ or in a finite subset of $\mathbb{U}$, then  they have  a unique joint Furstenberg system,\footnote{This is equivalent to the statement that all sequences of the form   $(\prod_{j=1}^m g_j(n+h_j))_{n\in\N}$ have logarithmic averages,  where $g_1,\ldots, g_m \in \{f_1,\ldots, f_\ell, \overline{f_1}, \ldots, \overline{f_\ell}\}$ and $m, h_1,\ldots, h_m\in \N$ are arbitrary.}  which   is ergodic and isomorphic  to
	the direct product of a procyclic system
		and  a Bernoulli system.
  \end{conjecture}
This  generalizes \cite[Conjecture~1]{FH18}, which concerned Furstenberg systems  of a single multiplicative function $f\colon \N\to [-1,1]$.
 If we further  restrict to the case where
 $f$ takes values in $\{-1,1\}$, then
 we conjectured in  \cite[Conjecture~2]{FH18} that  $f$
 should have a unique Furstenberg system, which   is either an ergodic procyclic system  or a Bernoulli system. Combining  \cite[Theorem~1.7]{BKLR17}  and \cite[Theorem~6]{DD82} we get that  the first alternative holds  if $f$ is not aperiodic (this happens if and only if  $\D(f,\chi)<\infty$ for some Dirichlet character $\chi$, see terminology in Section~\ref{SS:aperiodic}).
We expect that the second alternative holds exactly   when   $f$ is aperiodic.
This is known to be the case conditionally to the
assumption that all Furstenberg systems of $f$ are ergodic \cite[Corollary~1.5]{Fr17}. Unconditionally, this is not even known
 for the Liouville function; it is equivalent to the logarithmically averaged variant of the Chowla conjecture.

Perhaps surprisingly,  multiplicative functions with values on the unit circle may have non-ergodic Furstenberg systems, in fact, with uncountably many ergodic components. Consider  for instance the multiplicative function  $f(n):=n^{it}$, $n\in \N$,  for some non-zero $t\in \R$. We claim that  it has a unique Furstenberg system $(X,\mu,T)$, which is isomorphic to the system $(\T,m_\T,R)$, where $R$ is the identity transformation on $\T$. Indeed, let $G_0(y):=e^{2\pi i y}$, $y\in \T$, and $X:=\mathbb{U}^\Z$, $T$ be the shift transformation on $X$, $F_0(x):=x(0)$, $x\in X$,  and $\mu:=\lim_{N\to\infty}\frac 1{\log N}\sum_{n=1}^N\frac{ \delta_{f(n)}}{n}$  (we show below that the weak-star limit exists).
We claim that $\Phi\colon \T\to X$,  defined by $\Phi(y):=(e^{2\pi i y})_{n\in\Z}$, $y\in \T$,  is an isomorphism
between the systems $(\T,m_\T,R)$ and  $(X,\mu,T)$. The map $\Phi$ is clearly one to one and satisfies
$T\circ \Phi=\Phi\circ R$. It remains to show that  $\mu=m_\T\circ\Phi^{-1}$.
Notice first, that  due to the slowly varying nature of $n^{it}$, for fixed $h\in \Z$ we have  $(n+h)^{it}-n^{it}\to 0$ as $n\to\infty$. Using this and the fact that $\lim_{N\to\infty}\frac 1{\log N}\sum_{n=1}^N\frac{n^{it}}{n}=0$ for $t\neq 0$, we get that for every $m\in \N$ and $k_{-m},\ldots, k_m\in\Z$, we have
$$
\int_X\prod_{j={-m}}^m T^{h_j}F_0^{k_j}\, d\mu=\lim_{N\to\infty}\frac 1{\log N}\sum_{n=1}^N\frac{\prod_{j={-m}}^m f^{k_j}(n+h_j)}{n}=\int_{\T} \prod_{j={-m}}^m T^{h_j}G_0^{k_j}\, dm_\T,
$$
since the second and third terms  are  either simultaneously $0$ or $1$ depending on whether   $\sum_{j=-m}^mk_j\neq 0$ or $\sum_{j=-m}^mk_j= 0$.  Using this identity and the fact that $G_0=F_0\circ \Phi$ we
get that  $\mu=m_\T\circ\Phi^{-1}$, completing the proof that $\Phi$ is an isomorphism.

Similarly, if $f(n):=n^{it}\bmu(n)$, $n\in\N$, where $t\neq 0$ and $\bmu$ is the M\"obius function, then  we expect (but cannot prove) that $f$
has a unique Furstenberg system with uncountably many ergodic components, all of them isomorphic to a direct
product of a non-trivial procyclic system and a non-trivial Bernoulli system. It seems  likely that a similar structural result holds for general multiplicative functions with values on the unit disc:
\begin{conjecture}
A joint Furstenberg system of any multiplicative functions $f_1,\ldots, f_\ell\colon \N\to \mathbb{U}$
 has ergodic components  isomorphic to direct products of procyclic systems and Bernoulli systems.
\end{conjecture}

\subsection{Notation} For readers convenience, we gather here   some notation that we use frequently
 throughout the article.

We write $\T$ for the unit circle,  which we often identify with $\R/\Z$, and we write  $\mathbb{U}$ for the complex unit disc.
We denote by $\N$ the set of positive integers,  by $\P$ the set of prime numbers, and
for  $d\in \N$ we denote by  $\P_d$ the set $\P\cap(d\N+1)$.
For $N\in \N$ we denote by $[N]$ the set
 $\{1,\ldots,
N\}$. Whenever we write $\bN$ we mean a sequence of intervals of integers $([N_k])_{k\in\N}$ with $N_k\to\infty$.

\subsection{Acknowledgement} We would like to thank    M.~Lema\'nczyk  for the observation that the  convergence in Theorem~\ref{T:Sarnak1} is uniform.
We also thank M.~Lema\'nczyk and T. de la Rue for pointing out a correction in Theorem 1.4 and Part~(\ref{it:disjoint-2})  of  Proposition 5.1.

\section{Background in ergodic theory and number theory}

\subsection{Notation regarding averages}\label{SS:def}
 If $A$ is a non-empty finite subset of $\N$  we let
$$
\E_{n\in A}\,a(n):=\frac{1}{|A|}\sum_{n\in A}\, a(n), \quad
\lE_{n\in A}\,a(n):=\frac{1}{\sum_{n\in A}\frac{1}{n}}\sum_{n\in A}\frac{a(n)}{n}.
$$
If $A$ is an infinite subset of $\N$ we let
$$
\E_{n\in A}\, a(n):=\lim_{N\to\infty} \E_{n\in A\cap [N]}\, a(n), \quad
\lE_{n\in A}\, a(n):=\lim_{N\to\infty} \lE_{n\in A\cap [N]}\, a(n)
$$
if  the limits exist.
Let $\bN= ([N_k])_{k\in\N}$ be a sequence of intervals with $N_k\to \infty$.
We let
$$
\E_{n\in\bN}\, a(n):=\lim_{k\to\infty}\E_{n\in[N_k]} \, a(n), \quad
\lE_{n\in\bN}\, a(n):=\lim_{k\to\infty}\lE_{n\in[N_k]} \, a(n)
$$
if  the limits exist. Using partial summation one sees that if
$\E_{n\in \N}\,a(n)=0$, then also $\lE_{n\in \N}\,a(n)=0$. The converse does not hold in general, and the forward implication fails if one replaces $\N$ with general sequences
$\bN= ([N_k])_{k\in\N}$  with $N_k\to \infty$.

\subsection{Measure preserving systems}\label{SS:mps}
Throughout the article, we make the standard assumption that all probability  spaces $(X,\CX,\mu)$ considered are Lebesgue, meaning, $X$  can be given the structure of a compact metric space
and $\CX$ is its Borel $\sigma$-algebra.
A {\em measure preserving system}, or simply {\em a system}, is a quadruple $(X,\CX,\mu,T)$
where $(X,\CX,\mu)$ is a probability space and $T\colon X\to X$ is an invertible, measurable,  measure preserving transformation.
We typically omit the $\sigma$-algebra $\CX$  and write $(X,\mu,T)$. Throughout,  for $n\in \N$ we denote  by $T^n$   the composition $T\circ  \cdots \circ T$ ($n$ times) and let $T^{-n}:=(T^n)^{-1}$ and $T^0:=\id_X$. Also, for $f\in L^1(\mu)$ and $n\in\Z$ we denote by  $T^nf$ the function $f\circ T^n$.

In order to avoid unnecessary repetition,  we refer the reader to the article \cite{FH18} for some other standard notions from ergodic theory.
 In particular, the reader will find in Section~2 and
in  Appendix~A of \cite{FH18} the definition
of the terms factor, Kronecker factor,  isomorphism, inverse limit, spectrum, rational and irrational spectrum,  ergodicity, ergodic components, total ergodicity, nilsystem, infinite-step nilsystem, Bernoulli system, joining, and disjoint systems; all these notions  are used subsequently.

\subsection{Furstenberg systems associated with  several sequences}\label{SS:Furst}
To each finite collection of sequences $a_1,\ldots, a_\ell\colon \N\to \mathbb{U}$ that are distributed ``regularly'' along a
 sequence of intervals,
we  associate a measure preserving system defined using the joint distribution  of the sequences $a_1,\ldots, a_\ell$. For the purposes of this  article,
all averages in the definition of joint Furstenberg systems are taken to be logarithmic.
\begin{definition}\label{D:correlations}
	Let $ \bN:=([N_k])_{k\in\N}$ be a sequence of intervals with $N_k\to \infty$.
	We say that the sequences $a_1,\ldots, a_\ell\colon \Z \to \mathbb{U}$  {\em admit log-correlations  on $\bN$}, if the   limits
	\begin{equation}\label{E:wlim}
	\lim_{k\to\infty} \lE_{n\in [N_k]}\,  \prod_{j=1}^m b_j(n+h_j)
	\end{equation}
	exist for all $m \in \N$, all   $ h_1,\ldots, h_m\in \Z$,  
	and all $b_1,\ldots, b_m\in \{a_1,\ldots, a_\ell, \overline{a_1}, \ldots, \overline{a_\ell}\}$.
\end{definition}
\begin{remarks}
	$\bullet$	Given $a_1,\ldots, a_\ell\colon \Z \to \mathbb{U}$,  using a diagonal argument, we get that every sequence of intervals $\bN=([N_k])_{k\in \N}$
	has a subsequence $\bN'=([N_k'])_{k\in\N}$, such that the sequences  $a_1,\ldots, a_\ell$  admit log-correlations on $\bN'$.

	$\bullet$ If the sequences $a_1, \ldots, a_\ell$ are only defined on $\N$, then  we extend them in an arbitrary way to $\Z$ and give analogous definitions. Then all the limits in \eqref{E:wlim} do not depend on the choice of the extension.
\end{remarks}

\begin{definition}
	Let $(X,T)$ be a topological dynamical system. We say that the collection of
 functions $F_1,\ldots, F_\ell\in C(X)$ is \emph{$T$-generating} if the functions $T^nF_1,\ldots, T^nF_\ell$, $n\in\Z$, separate points  of $X$.
\end{definition}
\begin{remark}
By  the Stone-Weierstrass theorem, the functions $F_1,\ldots, F_\ell\in C(X)$ are $T$-generating  if and only if the $T$, $T^{-1}$-invariant subalgebra  generated by $F_1,\ldots, F_\ell$ and $\overline{F_1},\ldots,\overline{F_\ell} $
is dense in $C(X)$ with the uniform topology.
\end{remark}
We use the following  variant of the correspondence principle of Furstenberg~\cite{Fu77, Fu}
that applies to finite  collections of  bounded sequences of complex numbers:

\begin{proposition}\label{P:correspondence}
	Let $a_1,\ldots, a_\ell\colon \Z\to \mathbb{U}$ be sequences that  admit
	log-correlations  on
	$\bN:=([N_k])_{k\in\N}$.
	Then there exist a  topological dynamical  system $(X,T)$,  a    $T$-invariant Borel probability measure $\mu$, and a  $T$-generating collection of functions $F_{0,1}\ldots, F_{0,\ell}\in C(X)$,
	such that
	\begin{equation}\label{E:correspondence}
	\lE_{n\in {\bN} }\,  \prod_{j=1}^m b_j(n+h_j) =\int_X \prod_{j=1}^m
	T^{h_j}F_j \, d\mu
	\end{equation}
	for all  $m\in \N$, all  $h_1, \ldots, h_m\in \Z$, and all
	$b_j,\ldots, b_m\in\{a_1,\ldots, a_\ell, \overline{a_1},\ldots, \overline{a_\ell}\}$,
	where for $j=1,\ldots, m,$ if  the sequence
	$b_j$ is   $a_k$ or $\overline{a_k}$ for some $k\in \{1,\ldots, \ell\}$,
	then  $F_j$ is  $F_{0,k}$ or $\overline{F_{0,k}}$ respectively.
\end{proposition}
\begin{remark}
	In the arguments that follow we often use the explicit choice
	of $X$ and $T$ made  in the proof below, namely, we take $X=(\mathbb{U}^\ell)^\Z$ and let $T$ be the shift transformation on $X$. We also often assume that the functions  $F_{0,1}, \ldots, F_{0,\ell}$ are defined by \eqref{E:F0} below.
	\end{remark}
\begin{proof}
  Let $X:=(\mathbb{U}^\ell)^\Z$  and $T$ be the shift transformation on
	$X$, namely, $T$ maps an element $((x_1(n),\ldots, x_\ell(n)))_{n\in\Z}$ of $X$ to
	 $((x_1(n+1),\ldots, x_\ell(n+1)))_{n\in\Z}$.
	For $j=0,\ldots, \ell$ we define the functions $F_{0,j}\in C(X)$ as follows
	\begin{equation}\label{E:F0}
	F_{0,j}(x):=x_j(0), \quad \text{for } x=((x_1(n),\ldots, x_\ell(n)))_{n\in\Z}\in X.
	\end{equation}
	 Finally, the measure $\mu$  is defined to be the weak-star limit of the sequence of measures
	$\lE_{n\in[N_k]}\delta_{T^na}$, $k\in \N$,
	where $a:=((a_1(n),\ldots, a_\ell(n)))_{n\in\Z}\in X$. Then $\mu$ is $T$-invariant and  we have $F_{0,j}(T^na)=a_j(n)$, $n\in \Z$, for $j=1,\ldots, \ell$. It follows that
	\eqref{E:correspondence} holds and the proof is complete.
	\end{proof}

\begin{definition}
	Let $a_1,\ldots, a_\ell\colon \Z\to \mathbb{U}$ be  sequences that  admit
	log-correlations on
	$\bN:=(N_k)_{k\in\N}$.
	We call the system  (or the measure $\mu$)
	defined in Proposition~\ref{P:correspondence}
	the {\em joint Furstenberg system  (or measure)
		associated with $a_1,\ldots, a_\ell$ and $\bN$}.
\end{definition}
\begin{remarks}
	$\bullet$  Given $a_1,\ldots,a_\ell\colon \Z\to \mathbb{U}$ and $\bN$,  the measure  $\mu$ is uniquely determined by \eqref{E:correspondence} since this identity determines the values of $\int f\, d\mu$ for all real valued  $f\in C(X)$.
	
	$\bullet$
		If two or more sequences coincide, say for example that  $a_m=\cdots=a_\ell$ for some $m\in\{1,\ldots, \ell-1\}$, then
		it is not hard to see that the joint  Furstenberg system associated with  $a_1,\ldots, a_\ell$ and  $\bN$ is isomorphic with  the one associated with
		 $a_1,\ldots, a_m$ and $\bN$.

	$\bullet$ 	A collection of sequences $a_1,\ldots, a_\ell\colon \Z\to \mathbb{U}$ may have several non-isomorphic joint Furstenberg systems
	depending on which sequence of intervals $\bN$ we use in the evaluation of  their joint correlations.
	When we write that  a joint Furstenberg measure or system of the sequences  $a_1,\ldots, a_\ell$  has a certain property, we mean that any of these measures or systems has the asserted property.
\end{remarks}

\subsection{Convergence results} Henceforth, we use the following notation:
\begin{definition}
	If $d\in \N$ we let $\P_d:=\P\cap(d\N+1)$.
\end{definition}
We will use
 the following convergence result that
  was proved in  \cite{WZ11} and also in \cite{FHK}
conditional to some conjectures obtained later in \cite{GT7, GTZ12}:
\begin{theorem}
	\label{T:FHK}
	Let $(X,\mu,T)$ be a system and $d\in \N$. Then
	for every $\ell\in\N$ and  $F_1,\ldots ,F_\ell\in L^\infty(\mu)$ the following limit exists in $L^2(\mu)$
	$$
	\E_{p\in\P_d}\prod_{j=1}^\ell T^{pj}F_j.
	$$
\end{theorem}
\begin{remark}
Convergence is proved in  \cite{WZ11} and  \cite{FHK} for $d=1$.
The more general statement follows by using the $d=1$ case for product systems of the form
$T\times R$ acting on $X\times \Z/d\Z$ with the product measure,
 where $R$ is the translation by $1$ on $\Z/d\Z$, and for the functions
$F_1\otimes {\bf 1}_{d\Z+1}, F_2,\ldots, F_\ell$; of course, one also uses the fact that the relative density of the set $\P_d$ in the primes exists.
\end{remark}

 We will make use of  the following consequence of Theorem~\ref{T:FHK}:
\begin{proposition}	\label{P:conv}
	Suppose that the sequences $a_1,\ldots, a_\ell\colon \Z\to \mathbb{U}$
	admit log-correlations on the sequence of intervals $\bN$. Then for every $d\in \N$  the limit
	$$
	\E_{p\in\P_d}\Bigl(\lE_{n\in\bN}\prod_{j=1}^m b_j(n+ph_j)\Bigr)
	$$
	exists  for all  $m \in \N$, all  $h_1,\ldots, h_m\in \Z$, and all $b_1,\ldots,b_{m}\in \{a_1,\ldots, a_\ell, \overline{a_1},\ldots, \overline{a_\ell}\}$.
\end{proposition}
\begin{proof}
	Let $(X,\CX, \mu,T)$ be the joint Furstenberg system associated with  $a_1,\ldots, a_\ell$ and  $\bN$, and let  also $F_{0,1},\ldots, F_{0,\ell}\in L^\infty(\mu)$ be as in Proposition~\ref{P:correspondence}.
	Using Theorem~\ref{T:FHK} we get  that  the limit
	$$
	\E_{p\in\P_d}\int_X\prod_{j=1}^m T^{ph_j}F_j\,d\mu
	$$
	exists for all
	  $m\in\N$,  all  $h_1,\ldots, h_m\in \Z$,  and all  $F_1,\ldots,F_{m}\in L^\infty(\mu)$. Combining this with identity   \eqref{E:correspondence}
we get the asserted convergence.
\end{proof}

\subsection{Aperiodic and strongly aperiodic multiplicative functions}\label{SS:aperiodic}
We denote by  $\mathcal{M}$ the set of all multiplicative functions $f\colon \N\to \mathbb{U}$, where $\mathbb{U}$ is the complex unit disc.
A \emph{Dirichlet character} is a periodic
completely multiplicative function $\chi$  with $\chi(1)=1$.	
We say that   $f\in \mathcal{M}$ is \emph{aperiodic} (or \emph{non-pretentious} using terminology from \cite{GS16}) if it averages to $0$ on every infinite arithmetic progression, meaning,
if
$$\E_{n\in \N}\, f(an+b)=0, \quad \text{  for all } a,b\in \mathbb{N}.
$$
This is equivalent to asserting that $\E_{n\in\N}\, f(n)\, d(n)=0$ for every periodic sequence $d\colon \N\to \C$, or that
$\E_{n\in \N}\, f(n)\, \chi(n)=0$ for every Dirichlet character $\chi$. In order to give easier to verify necessary conditions for aperiodicity,  we use a notion of   distance between two multiplicative functions defined as in  \cite{GS16}:

\begin{definition}
	We let $\D\colon \mathcal{M}\times \mathcal{M}\to [0,\infty]$ be given by
	$$
	\D(f,g)^2:=\sum_{p\in \P} \frac{1}{p}\,\big(1-\Re\big(f(p) \overline{g(p)}\big)\big)
	$$
	where $\Re(z)$ denotes the real part of a complex number $z$.
\end{definition}

It is shown in \cite[Theorem~1]{D83} that  $f\in \mathcal{M}$ is aperiodic if and only if
 $\D(f, \chi \cdot n^{it})=\infty$ for every $t\in \R$ and every  Dirichlet character $\chi$.
 Moreover, if $f$ takes real values, then $f$ is aperiodic if  and only if
 $\D(f, \chi)=\infty$ for every Dirichlet character $\chi$. In particular, the M\"obius and the Liouville functions are aperiodic.

For our purposes we also  need a notion introduced in \cite{MRT15}
that  is somewhat stronger than aperiodicity.
\begin{definition}\label{D:StronglyAperiodic}
 Let $\D\colon \mathcal{M}\times \mathcal{M}\times \mathbb{N} \to [0,\infty]$ be given by
	$$
	\D(f,g;N)^2:=\sum_{p\in \P\cap [N]} \frac{1}{p}\,\bigl(1-\Re\bigl(f(p) \overline{g(p)}\bigr)\bigr)
	$$
	and $M\colon \mathcal{M}\times \mathbb{N} \to [0,\infty)$ be given by
	$$
	M(f;N):=\min_{|t|\leq N} \D(f, n^{it};N)^2.
	$$
	The multiplicative function $f\in \mathcal{M}$ is
	{\em strongly  aperiodic}
	if $M(f\cdot\chi;N)\to \infty$ as $N
	\to \infty$ for every Dirichlet character $\chi$.
\end{definition}
Note  that strong aperiodicity implies aperiodicity. The converse is
not in general true (see  \cite[Theorem~B.1]{MRT15}), but it is
true for  real valued  multiplicative functions (see  \cite[Appendix~C]{MRT15}). In particular,  the    M\"obius and the Liouville functions are strongly aperiodic.
Furthermore,  if $f\in \mathcal{M}$ is aperiodic  and  $f(p)$ is a $d$-th root of unity
for all  $p\in \P$,
then $f$ is strongly aperiodic (see \cite[Proposition~6.1]{F16}). In particular, if  $f(p)$ is a nontrivial $d$-th root of unity for all $p\in\P$, then $f$ is strongly aperiodic (see \cite[Corollary~6.2]{F16}).

The hypothesis of strong aperiodicity is useful for our purposes because it enables us to use the following result of Tao~\cite[Corollary~1.5]{Tao15}:
\begin{theorem}\label{T:Tao}
Let $f\in \mathcal{M}$ be a strongly aperiodic multiplicative function.
Then we have
	$$
	\lE_{n\in \N}\,  f(n)\, \overline{f(n+h)}=0
$$
 for every $h\in \N$.
\end{theorem}
\begin{remark}
	By adjusting the example in  \cite[Theorem~B.1]{MRT15}, it follows  that strong aperiodicity cannot be replaced by aperiodicity; in particular,  there exist  an aperiodic multiplicative function $f\in \mathcal{M}$, a positive constant $c$,  and a sequence of intervals ${\bf N}:=([N_k])_{k\in\mathbb{N}}$ with $N_k\to \infty$, such that
	$$
	|\lE_{n\in {\bf N}} \, f(n)\cdot \overline{f(n+h)} |\geq c, \quad \text{ for every } h\in \mathbb{N}.
	$$	
\end{remark}

\section{Correlation identities and ergodic consequences}\label{SS:Tao}
\subsection{Correlation  identities} If $a\colon \P\to \mathbb{U}$ is a sequence and $A$ is a non-empty finite or infinite subset of the primes we define $\lE_{p\in A}$ as in Section~\ref{SS:def}.
The following identity of Tao and Ter\"av\"ainen from \cite[Theorem~3.6]{TT17}
 is key for our purposes:
\begin{theorem}\label{T:Tao1}
	Suppose that the multiplicative functions $f_1,\ldots, f_\ell \colon \Z\to \mathbb{U}$  admit log-correlations on the sequence of intervals $\bN$.   Then  we have
	\begin{equation}\label{E:TT}
	\lE_{p\in \P}\Big|c_{p,m}\, \lE_{n\in \bN}\,  \prod_{j=1}^m g_j(n+h_j)-
	\, 	\lE_{n\in \bN}\,  \prod_{j=1}^{m} g_j(n+ph_j)\Big|=0
	\end{equation}
	for all $m \in \N$, all  $h_1,\ldots, h_m\in \Z$, and all $g_1,\ldots,g_{m}\in \{f_1,\ldots, f_\ell, \overline{f_1},\ldots, \overline{f_\ell}\}$, where $c_{p,m}:=\prod_{j=1}^m g_j(p)$, $p\in \P$.
\end{theorem}
\begin{remarks}
$\bullet$ A variant of this result  is also implicit in the article of Tao \cite{Tao15} and was also
used in \cite{FH18} for $f_1=\cdots=f_\ell=\mu$ or $\lambda$ using a different averaging scheme. The current version is more suitable for our purposes.

$\bullet$	In \cite{TT17} the result is proved for  a class of generalized limit functionals in
place of 	$\lE_{n\in \bN}$. Assuming that $\bN=([N_k])_{k\in\N}$, the asserted version follows if one  uses a generalized limit functional of the form $\wt\lim_{k\to\infty}\lE_{n\in {[N_k]}}$ since it coincides with the standard limit $\lim_{k\to\infty}\lE_{n\in {[N_k]}}=\lE_{n\in\bN}$
whenever this limit exists.
	\end{remarks}

For the record, we mention the following  identity for general sequences which follows from the proof  of \cite[Theorem~3.6]{TT17} without any essential change; Theorem~\ref{T:Tao1} is an easy consequence of this identity:
\begin{theorem}
	Let $\bN$ be  a sequence of intervals,   $a_1,\ldots, a_\ell\colon \Z\to \mathbb{U}$  be sequences,   and   $h_1,\ldots, h_\ell\in \Z$.
	Then,
	assuming that   for every $p\in\P$ the limits $\lE_{n\in \bN}$ below exist,
	we have the identity
$$
	\lE_{p\in \P}
	\Big|\lE_{n\in \bN}\,\prod_{j=1}^\ell a_j(pn+ph_j)-
	\lE_{n\in \bN}\, \prod_{j=1}^\ell a_j(n+ph_j)\Big|=0.
$$
\end{theorem}

\subsection{A consequence of the correlation identities }
We are going to combine Theorem~\ref{T:Tao1} with  Theorem~\ref{T:TT} stated below  in order to prove
a variant of the identity \eqref{E:TT}  in which the weights $c_{p,m}$ are all equal to $1$.
For convenience we introduce the following notation:
\begin{definition} Let $a,b\colon \P\to \mathbb{U}$ be sequences.  We write
	$a\sim b$ if $$
	\lE_{p\in \P}(1-\Re(a(p)\cdot \overline{b(p)}))=0.
	$$
\end{definition}
\begin{remarks}
$\bullet$	If we restrict to sequences that take values on the unit circle, then $\sim$ is an equivalence relation and $a\sim b$ is equivalent to $\lE_{p\in \P}|a(p)-b(p)|=0$.

$\bullet$ Using terminology from  \cite{TT17} we  have  that two multiplicative functions $f,g\colon \Z\to\mathbb{U}$ satisfy  $f\sim g$ exactly when   ``$f$ weakly pretends to be $g$''.
\end{remarks}

We will use the following basic properties:
\begin{lemma}\label{L:sim}
If  $a,b,c,d\colon \P \to \mathbb{U}$ are sequences, then the following properties hold:
\begin{enumerate}
	\item If $a\sim b$, then $\overline{a}\sim \overline{b}$.
	
		\item If $a\sim b$ and $b\sim c$, then  $a \sim c$.
	
	\item If $a\sim b$ and $c\sim d$, then  $a c\sim b d$.
	
	\item If $a\sim b$, then $\lE_{p\in \P}|a(p)-b(p)|=0$.
\end{enumerate}
\end{lemma}
\begin{proof}
Property $\text{(i)}$ is obvious. Properties  $\text{(ii)}$ and $\text{(iii)}$ follow from the estimate
\begin{equation}\label{E:uv}
1-\Re(uv)\leq 2(1-\Re(u)+1-\Re(v))
\end{equation}
 which holds for all $u,v\in \mathbb{U}$. One way to prove this, is to first consider the case where $|u|=|v|=1$; in this case  the estimate is equivalent to $|u-v|^2\leq 2(|1-u|^2+|1-v|^2)$, which follows from the Cauchy Schwarz inequality. One then deduces from this the general case by expressing arbitrary $u,v\in \mathbb{U}$ as  a convex combination of two elements on the unit circle and taking advantage of the linearity features of \eqref{E:uv}.
  Property $\text{(iv)}$ follows from the estimate
$$
|u-v|^2\leq 2(1-\Re(u\overline{v}))
$$ which holds for all $u,v\in \mathbb{U}$.
\end{proof}

We will use the next result of  Tao and  Ter\"av\"ainen   \cite[Theorem 1.2]{TT17}:
\begin{theorem}\label{T:TT}
	Let $f_1,\ldots, f_\ell\colon \Z\to\mathbb{U}$ be multiplicative functions.  Suppose that
for every Dirichlet character $\chi$ we have 	$f_1\cdots f_\ell\nsim \chi$. Then
$$
\lE_{n\in\N}\,  \prod_{j=1}^\ell f_j(n+h_j)=0
$$
for all $h_1,\ldots, h_\ell\in \Z$.
\end{theorem}
\begin{remarks}
$\bullet$	The use of logarithmic averages is essential for the statement to hold. For example,
	take $\ell=1$ and let  $f_1(n):=n^{it}$, $n\in\N$, for some non-zero  $t\in \R$. Then $f_1 \nsim\chi$ for every Dirichlet character $\chi$ but  the limit
	$\lim_{N\to\infty}\E_{n\in [N]} \, n^{it}$ does not exist (since $\E_{n\in [N]}\, n^{it}=\frac{N^{it}}{it+1}+o(1)$). On the other hand we have that $\lE_{n\in\N}\,  n^{it}=0$.

$\bullet$	 The proof of Theorem~\ref{T:TT} depends crucially on deep results from ergodic theory such as
\cite{HK,L15a,L15b} and analytic number theory  \cite{GT08,GT09b}. 	
\end{remarks}
The next result is a key ingredient in our argument (recall that $\P_d=\P\cap (d\N+1)$):
\begin{proposition}\label{P:constantone}
		Let $f_1,\ldots, f_\ell\colon \Z\to\mathbb{U}$ be multiplicative functions. There exists $d\in \N$ such that the following holds:
	If $f_1,\ldots, f_\ell$
	admit log-correlations on the sequence of intervals $\bN$, then
	\begin{equation}\label{E:noconstant}
	\lE_{p\in \P_d}\Big|\lE_{n\in \bN}\,  \prod_{j=1}^{m} g_j(n+h_j)-
	\, \lE_{n\in \bN}\,  \prod_{j=1}^{m}g_j(n+ph_j)\Big|=0
	\end{equation}
	for all $m \in \N$, all  $h_1, \ldots, h_{m} \in \Z$, and all $g_1,\ldots,g_{m}\in \{f_1,\ldots, f_\ell, \overline{f_1},\ldots, \overline{f_\ell}\}$.
\end{proposition}
\begin{proof}
	Suppose first that for some $j\in \{1,\ldots, \ell\}$ we have  $\lE_{n\in\bN}|f_j(n)|=0$. Then
	 whenever one of the functions $g_1, \ldots, g_m$ is equal to $f_j$ or $\overline{f_j}$,
	 all logarithmic averages  in  \eqref{E:noconstant} vanish and the identity holds trivially for $d=1$. Thus, without loss of generality, we can assume that for  $j=1,\ldots, \ell$ we have  $\lE_{n\in \bN}|f_j(n)|> 0$ (note that by our assumptions the average exists).
	 	Using Theorem~\ref{T:Tao1} for $\ell=1$, $g_1=f_j$,  and  $h_1=0$,
	we deduce that $\lE_{p\in \P}(1-|f_j(p)|)=0$.\footnote{We note that the $\ell=1$ case of Theorem~\ref{T:Tao1} admits a simple elementary proof via the Tur\`an-Kubilius inequality.} Hence, we can work under the assumption that
	\begin{equation}\label{E:norm1}
	|f_j|\sim 1, \quad \text{for  } \, j=1,\ldots, \ell.
	\end{equation}

	Next,   for $k\in \N$ and $j=1,\ldots, \ell$, we denote  by $f_j^{-k}$ the function $\overline{f_j^k}$ and  let $K=K_{f_1,\ldots, f_\ell}$ be the subset of $\Z^\ell$ defined as follows
	$$
	K:=\big\{(k_1,\ldots, k_\ell)\in \Z^\ell\colon \prod_{j=1}^\ell f_j^{k_j}\sim \chi \text{ for some Dirichlet character } \chi \big\}.
	$$
	Using \eqref{E:norm1} and  properties $\text{(i)}$-$\text{(iii)}$ of Lemma~\ref{L:sim}, and since products and complex conjugates of Dirichlet characters are Dirichlet characters,  we get that  $K$ is a subgroup of $\Z^\ell$. Since every subgroup of $\Z^\ell$ is finitely generated,  $K$  is finitely generated.
	We  let
	$F_K=F_{K,f_1,\ldots, f_\ell}$  be the following set  of multiplicative functions
	$$
	F_K:=\big\{\prod_{j=1}^\ell f_j^{k_j}\colon  (k_1,\ldots, k_\ell)\in K\big\}.
	$$
	We have that  $F_K$ is  finitely generated under multiplication. Let    $\{f_{0,1},\ldots, f_{0,r}\}$, for some $r\in\N$,   
	be a set of  generators for $F_K$. Then for $j=1,\ldots, r$ there exist Dirichlet characters $\chi_j$ such that $f_{0,j}\sim \chi_j$.
	If $d\in \N$ is a common period of all these Dirichlet characters,  then for $j=1,\ldots,r$ we have  $\chi_j(dn+1)=1$ for every $n\in \N$.
	Let  $f\in F_K$, then  $f=\prod_{j=1}^rf_{0,j}^{k_j}$ for some $k_1,\ldots, k_r\in \Z$. Since $f_{0,j}\sim \chi_j$  for $j=1,\ldots, r,$  we get from property $\text{(iii)}$  of Lemma~\ref{L:sim} that  $f\sim \prod_{j=1}^r\chi_j^{k_j}$, and since  $\chi_j(p)=1$ for all $j\in \{1,\ldots, r\}$  and all $p\in \P_d$, we deduce from property $\text{(iv)}$
	of Lemma~\ref{L:sim} that $\lE_{p\in\P_d}|f(p)-1|=0$ (we also used that $\P_d$ has positive relative density in $\P$).
	Hence,
	\begin{equation}\label{E:Pd}
	\lE_{p\in\P_d}|f(p)-1|=0, \quad \text{ for every }  f\in F_K.
	\end{equation}

	We  now  show that  \eqref{E:noconstant} holds. Let $\wt g:=\prod_{j=1}^mg_j$. Since $g_j\in \{f_1,\ldots, f_\ell, \overline{f_1},\ldots, \overline{f_\ell}\}$ for $j=1,\ldots, r$,
 using \eqref{E:norm1} and  properties $\text{(ii)}$ and $\text{(iii)}$ of Lemma~\ref{L:sim},
	we get that  $\wt g\sim\prod_{j=1}^\ell f_j^{k_j}$ for some $k_1,\ldots, k_\ell\in \Z$, where we continue to use the notation  $f^k$  for  $\overline{f^{-k}}$ if $k$ is a negative integer.
	We consider two cases.  If  $\wt g \notin F_K$, then $\wt g\nsim \chi$ for all Dirichlet characters $\chi$, in which case   \eqref{E:noconstant} holds  because by Theorem~\ref{T:TT} we have $ \lE_{n\in \bN}\prod_{j=1}^{m} g_j(n+h_j)=0$ for all $h_1,\ldots, h_{m}\in \Z$. On the other hand, if  $\wt g\in F_K$, we see that  \eqref{E:noconstant} holds by combining  Theorem~\ref{T:Tao1} (with $\P_d$ in place of $\P$) and \eqref{E:Pd}.
	This completes the proof.
\end{proof}

We are going to use the following consequence of Proposition~\ref{P:constantone} which is better suited for our purposes:

\begin{corollary}\label{C:constantone}
	Let $f_1,\ldots, f_\ell\colon \Z\to\mathbb{U}$ be multiplicative functions. There exists $d\in \N$ such that the following holds: If    $f_1,\ldots, f_\ell$
	admit log-correlations on the sequence of intervals $\bN$, then    the limit on the right hand side below exists and  we have
	\begin{equation}\label{E:noconstant'}
	\lE_{n\in \bN}\,  \prod_{j=1}^m g_j(n+h_j)=
	\E_{p\in \P_d}  \lE_{n\in \bN}\,  \prod_{j=1}^{m}g_j(n+ph_j)
	\end{equation}
	for all $m \in \N$, all  $h_1, \ldots, h_{m} \in \Z$, and all $g_1,\ldots,g_{m}\in \{f_1,\ldots, f_\ell, \overline{f_1},\ldots, \overline{f_\ell}\}$.
\end{corollary}
\begin{proof}
This follows immediately from Proposition~\ref{P:constantone} and the fact that by Proposition~\ref{P:conv} the limit $\E_{p\in \P_d}$ on the right hand side exists and as a consequence the limit $\E^*_{p\in \P_d}$ is equal to  the limit $\E_{p\in \P_d}$.
\end{proof}

\subsection{An ergodic consequence}
Using Proposition~\ref{P:correspondence} we deduce from  Corollary~\ref{C:constantone} the following ergodic result concerning joint Furstenberg systems of multiplicative functions:
\begin{theorem}\label{T:Tao2}
	Let $f_1,\ldots, f_\ell\colon \Z\to \mathbb{U}$  be multiplicative functions. There exists $d\in \N$ such that the following holds: If
	$(X,\mu,T)$ is a joint Furstenberg system  of $f_1,\ldots, f_\ell$, and if    $F_{0,1}, \ldots, F_{0,\ell}$ are as in Proposition~\ref{P:correspondence}, then we have
	\begin{equation}
	\label{eq:Furstenberg-Tao}
	\int_X \prod_{j=1}^{m} T^{h_j} F_j \, d\mu=   \E_{p\in \P_d}  \int_X \prod_{j=1}^{m} T^{ph_j}F_j\, d\mu
	\end{equation}
	for all $m \in \N$, all  $h_1,\ldots, h_{m}\in \Z$, and all
	$F_1,\ldots, F_{m}\in \{F_{0,1},\ldots F_{0,\ell}, \ \overline{F_{0,1}},\ldots, \overline{F_{0,\ell}}
	\}$.
\end{theorem}

\section{The structure of Furstenberg systems of  multiplicative functions}
The goal of this section is to prove Theorem~\ref{T:structure}. In the next section we use this structural result to prove Theorems~\ref{T:Sarnak1}-\ref{T:Sarnak2}.

\subsection{Proof of Theorem~\ref{T:structure}}
Given a system $(X,\mu,T)$ and $d\in \N$ we define the system of arithmetic progressions with steps in $\P_d$ as follows:
\begin{definition}\label{D:tilde}
	Let $(X, \mu,T)$ be a system  and let $X^\Z$ be endowed with the product $\sigma$-algebra.	For $d\in \N$ we write $\wt\mu_d$ for the probability  measure on $X^\Z$ defined as follows:
	For every $m\in\N$ and all $F_{-m},\ldots,F_m\in L^\infty(\mu)$, we let
	\begin{equation}
	\label{E:mud}
	\int_{X^\Z}\prod_{j=-m}^m F_j(x_j)\,d\wt\mu_d(\ux):=		
	\E_{p\in\P_d}\int_X\prod_{j=-m}^m T^{pj}F_j\,d\mu,
	\end{equation}
	where  $\ux:=(x_j)_{j\in\Z}$ and the limit on the right hand side  exists by Theorem~\ref{T:FHK}.
	The measure $\wt\mu_d$ is invariant under
	the shift transformation $S$ on $X^\Z$ and induces a system
	  $(X^\Z, \wt\mu_d,S)$, which we call the \emph{system of arithmetic progressions with steps in $\P_d$} associated with the system  $(X,\mu,T)$.
\end{definition}
\begin{remark}
 For $d=1$ the system $(X^\Z, \wt\mu_1,S)$ coincides with the system of arithmetic progressions with prime steps introduced in \cite[Definition~3.8]{FH18}.
	\end{remark}
The relevance of  the systems $(X^\Z, \wt\mu_d,S)$ to our problem is demonstrated  by  the following result:
\begin{proposition}
	\label{P:factor}
	Let $f_1,\ldots, f_\ell\colon \Z\to \mathbb{U}$ be multiplicative functions. Then there exists $d\in \N$ such that 	any joint  Furstenberg system $(X,\mu,T)$ of the  multiplicative functions $f_1,\ldots, f_\ell$ 	
	is a  factor of the system  $(X^\Z,\wt\mu_d,S)$.
\end{proposition}
\begin{proof}
	We  can assume that the joint  Furstenberg system is defined on the space   $X:=(\mathbb{U}^\ell)^\Z$ and  $T$ is the shift transformation on $X$. We denote elements of $X$ by $x=(x_1(k), \ldots, x_\ell(k))_{k\in\Z}$, where $x_1(k), \ldots, x_\ell(k)\in \mathbb{U}$ for $k\in \Z$, and  elements of $X^\Z$ with  $\ux=(x_n)_{n\in\Z}$, where $x_n\in X$ for $n\in \Z$. Hence, $\ux=(x_{n,1}(k), \ldots, x_{n,\ell}(k))_{k,n \in \Z}$ can be identified with $(x_{n,1},\ldots, x_{n,\ell})_{n\in \Z}$, where $x_{n,j}=(x_{n,j}(k))_{k\in\Z}$ for $j=1,\ldots, \ell$.
	
	We define the map $\pi\colon X^\Z\to X$  as follows:
	For
	$\ux=(x_{n, 1},\ldots, x_{n,\ell})_{n\in\Z}\in X^\Z$ let
	$$
	(\pi(\ux))(n):=( x_{n,1}(0),\ldots,  x_{n, \ell}(0))=(F_{0,1}(x_{n,1}),\ldots, F_{0,\ell}(x_{n, \ell})),  \quad n\in\Z,
	$$	
	where	
	$$
	F_{h, j}(x):=x_j(h),\quad  x\in X,\ h\in \Z,\  j\in \{1,\ldots, \ell\}.
	$$

	For $n\in\Z$ we  have
	\begin{multline*}
	(\pi(S\ux))(n)=(F_{0,1}((S\ux)_n), \ldots,F_{0,\ell}((S\ux)_n)) =(F_{0,1}(x_{n+1, 1}), \ldots,F_{0,\ell}(x_{n+1, \ell}))=\\ (\pi(\ux))(n+1)=(T\pi(\ux))(n).
	\end{multline*}
	Thus
	$$
	\pi\circ S=T\circ\pi.
	$$
	
	Next, we claim that  $\wt\mu_d\circ \pi^{-1}=\mu$.
	Indeed, for all $m \in\N$, all  $h_1,\dots,h_m\in\Z$, and all $k_1,\ldots, k_m\in \{\pm 1,\ldots, \pm\ell\}$, using   identity~\eqref{eq:Furstenberg-Tao} in Theorem~\ref{T:Tao2} and the definition of $\wt\mu_d$ given in~\eqref{E:mud}, we have
	\begin{multline*}
	\int_X\prod_{j=1}^m F_{h_j, k_j}(x)\, d\mu(x)=
	\int_X\prod_{j=1}^m F_{0, k_j}(T^{h_j}x)\,d\mu(x)
	=
	\E_{p\in\P_d} \int_X\prod_{j=1}^\ell F_{0, k_j}(T^{ph_j}x)\,d\mu(x)
	=\\
	\int_{X^\Z}\prod_{j=1}^m F_{0,k_j}(x_{h_j})\,d\wt\mu_d(\ux)
	=\int_{X^\Z}\prod_{j=1}^m (F_{h_j, k_j}\circ\pi)(\ux)\,d\wt\mu_d(\ux),
	\end{multline*}
	where we let $F_{h,-k}:=\overline{F_{h,k}}$ for $h\in \Z$ and $k\in \{1, \ldots, \ell\}$.
	Since the algebra generated by  the functions $F_{h,1}, \ldots,  F_{h,\ell}, \overline{F_{h,1}}, \ldots,  \overline{F_{h,\ell}}$, $h\in \Z$,  is dense in $C(X)$ with the uniform topology, the claim follows.
	
	Therefore, $\pi\colon (X^\Z,\wt\mu_d,S)\to (X,\mu,T)$ is a factor map and the proof is complete.
\end{proof}

Our next task is  to obtain structural results for  the systems $(X^\Z,\wt\mu_d,S)$. This crucially
depends on  the following result from \cite{FH18} which deals with the case where $d=1$:
\begin{theorem}\label{T:FH}
	Let $(X,\mu,T)$ be a system. Then the system $(X^\Z,\wt\mu_1,S)$
	\begin{enumerate}
		\item
		has no irrational spectrum;
		
		\item  has ergodic components
		isomorphic to   direct products of infinite-step nilsystems and  Bernoulli systems.	
	\end{enumerate}
\end{theorem}
\begin{remark} The   infinite-step nilsystems and the Bernoulli systems are allowed to be trivial.
\end{remark}

The proof of Theorem~\ref{T:FH} uses some deep ergodic machinery such as the main result from \cite{HK} (or \cite{Zi07}) regarding characteristic factors of Furstenberg averages, results about arithmetic progressions on nilmanifolds, and properties of partially  strongly stationary systems.
It  also uses indirectly (via the use of variants of limit formulas obtained in \cite{FHK}) some deep number theoretic input
such as the Gowers uniformity of the $W$-tricked von Mangoldt function from   \cite{GT09b, GT7, GTZ12}.
Luckily,  we do not have to modify the argument from \cite{FH18} in order to get a similar result for the measures $\wt\mu_d$; instead, we make use of the following simple observation, which allows us to use Theorem~\ref{T:FH} as a ``black box'':
\begin{lemma}\label{L:abscont}
Let $(X,\mu,T)$ be a system  and  $\wt\mu_d$, $d\in \N$, be the measures on $X^\Z$ defined by \eqref{E:mud}.
Then $\wt\mu_d\leq \phi(d) \, \wt\mu_1$ for every $d\in \N$.
\end{lemma}
\begin{proof}
It suffices to show that for all $m\in\N$ and all non-negative $F_{-m},\ldots, F_m\in L^\infty(\mu)$ we have
$$
\int_{X^\Z}\prod_{j=-m}^m F_j(x_j)\,d\wt\mu_d(\ux)	\leq \phi(d)\, \int_{X^\Z}\prod_{j=-m}^m F_j(x_j)\,d\wt\mu_1(\ux).		
$$
This follows immediately from \eqref{E:mud},  the fact  that the relative density $d_\P(\P_d)$ of the set $\P_d$ in the primes is $1/\phi(d)$, and the estimate
$$
\E_{p\in\P_d}\, a(p)\leq (d_\P(\P_d))^{-1}\, \E_{p\in\P} \, a(p)
$$
which holds for all  sequences $a\colon \P\to \R^+$ assuming that the limits on the left and right hand side  exist.
\end{proof}
Combining Theorem~\ref{T:FH} and Lemma~\ref{L:abscont} we  deduce the following:
\begin{theorem}\label{T:Structure-mud}
	Let $(X,\mu,T)$ be a system. Then for every   $d\in \N$ the system $(X^\Z,\wt\mu_d,S)$
	\begin{enumerate}
		\item
		has no irrational spectrum;
		
		\item  has ergodic components
		isomorphic to   direct products of infinite-step nilsystems and  Bernoulli systems.	
	\end{enumerate}
\end{theorem}
\begin{proof}
	By Lemma~\ref{L:abscont}  we have $\wt\mu_d\leq \phi(d)\, \wt\mu_1$, hence the measure $\wt\mu_d$ is absolutely continuous with respect to the measure $\wt\mu_1$. This implies that the spectrum of the system   $(X^\Z,\wt\mu_d,S)$ is a subset of the spectrum of the system
	$(X^\Z,\wt\mu_1,S)$,  so the former has no irrational spectrum since the same holds for the latter by
	Theorem~\ref{T:FH}. Furthermore, if $\wt\mu_1=\int_\Omega \wt\mu_{1,\omega}\, dP(\omega)$ is the ergodic decomposition of the measure
	$\wt\mu_1$, then the ergodic decomposition  of the measure $\wt\mu_d$  is  $\wt\mu_d=\int_\Omega \wt\mu_{1,\omega}\, dP_d(\omega)$
	for some probability measure $P_d$ that is absolutely continuous with respect to $P$. This implies that property $\text{(ii)}$ holds for the ergodic components of the measure $\wt\mu_d$ since it holds for the ergodic components of the measure $\wt\mu_1$ by Theorem~\ref{T:FH}.
\end{proof}

 Theorem~\ref{T:structure}  now follows by combining Proposition~\ref{P:factor} and Theorem~\ref{T:Structure-mud}.

  \subsection{An alternative proof of Theorem~\ref{T:structure} for some special cases}\label{SS:alternative}
  In some interesting special cases  we can prove Theorem~\ref{T:structure} (and hence Theorems~\ref{T:Sarnak1}-\ref{T:Sarnak2})
  using an alternative approach that avoids the  use of Theorem~\ref{T:TT}. We present the details below,
   let us emphasize though, that  this alternative approach   breaks down when
   $f_j(\P)$ is an infinite subset of the unit circle for some $j\in \{1,\ldots, \ell\}$, and
  we do not see how to avoid the use of Theorem~\ref{T:TT} in order to cover such cases.

\subsubsection{The case where $f_1(\P),\ldots, f_\ell(\P)$ are finite subsets of $\T$} Suppose  first that
 the multiplicative
  functions $f_1,\ldots, f_\ell$ are such  that  $f_j(\P)$ is a finite subset  of $\T$ for $j=1,\ldots, \ell$. Let  $(X,\mu,T)$ be a joint Furstenberg system associated with these  multiplicative functions  and a sequence of intervals $\bN$.
  Then there exist $c_1,\ldots, c_\ell\in \mathbb{U}$, a subset $A$ of $\P$,
  and a
  sequence of intervals  $\bM=([M_k])_{k\in\N}$ with $M_k\to\infty$, such that
  \begin{enumerate}
  	\item $f_j(p)=c_j$, $j=1\ldots, \ell$,  for all  $p\in A$;
  	
  	\item $d^{\log}_{\bM,\P}(A):=\lE_{\bM,p\in \P}\, {\bf 1}_A(p)$ exists and is positive;
  	
  	\item the averages
  	$$
  	\lE_{\bM, p\in A}\int_X\prod_{j=-m}^m T^{pj}F_j\,d\mu
  	$$
  	exist for all $m\in \N$ and $F_{-m},\ldots, F_m\in L^\infty(\mu)$,
  \end{enumerate}
where we used the notation
$$
\lE_{\bM, p\in A}\, a(p):=\lim_{k\to\infty}\lE_{p\in A\cap [M_k]}\, a(p)
$$
for $a\colon \P\to \C$, if  the limit exists.
  Using Theorem~\ref{T:Tao1} and property $\text{(i)}$, we get as in the proof of Proposition~\ref{P:factor}, using the factor map $\pi\colon X^\Z\to X$ defined by   $(\pi(\ux))(n):=(\overline{c_1} \, x_{n,1}(0),\ldots,  \overline{c_\ell}\,  x_{n, \ell}(0))$,
   that the system  $(X,\mu,T)$  is a factor of the system
  $(X^\Z,\mu^*,S)$ where  the measure $\mu^*$ is defined as follows:
  For every $m\in\N$ and all $F_{-m},\ldots,F_m\in L^\infty(\mu)$, we let
  \begin{equation}
  	\label{E:mu*}
  	\int_{X^\Z}\prod_{j=-m}^m F_j(x_j)\,d\mu^*(\ux):=		
  	\E^*_{\bM,p\in A}\int_X\prod_{j=-m}^m T^{pj}F_j\,d\mu,
  \end{equation}
where the limit on the right hand side exists by property $\text{(iii)}$.
  Since  $d^*_{\bM,\P}(A)>0$ (by property $\text{(ii)}$), we get that for every
  sequence $a\colon \P\to \R^+$ for which the limits below exist that
  $$
  \lE_{\bM,p\in A}a(p) \leq C\, \lim_{k\to\infty}\lE_{p\in \P\cap [M_k]}\, a(p)= C\,  \E_{p\in \P}\, a(p)
  $$
  where
  $C:=(d^{\log}_{\bM,\P}(A))^{-1}$ (the last identity holds because $\lim_{mkto\infty}\lE_{p\in \P\cap [M_k]}\, a(p)= \E_{p\in \P}\, a(p)$ since the last limit is assumed to exist).
  Using this estimate  in the case where  $a(p):=\int_X\prod_{j=-m}^m T^{pj}F_j\,d\mu$, $p\in\P$, is non-negative,    we get  that the measures $\wt\mu_1$ and $\mu^*$, defined by \eqref{E:mud} and \eqref{E:mu*} respectively, satisfy the estimate
  $$
  \int_{X^\Z}\prod_{j=-m}^m F_j(x_j)\,d\mu^*(\ux)\leq
  C\,  \E_{p\in \P}\int_X\prod_{j=-m}^m T^{pj}F_j\,d\mu= C \int_{X^\Z}\prod_{j=-m}^m F_j(x_j)\,d\wt\mu_1(\ux)
  $$
  for all $m\in\N$ and all non-negative $F_{-m},\ldots,F_m\in L^\infty(\mu)$.
  Hence, $\mu^*\leq C\, \wt \mu_1$, and as in the proof of Theorem~\ref{T:Structure-mud}
  we conclude that the system  $(X^\Z,\mu^*,S)$ satisfies properties $\text{(i)}$ and $\text{(ii)}$
  of Theorem~\ref{T:Structure-mud}. Since   $(X,\mu,T)$  is a factor of the system
  $(X^\Z,\mu^*,S)$, it also satisfies these two properties.

\subsubsection{The case of real valued multiplicative functions}
  Let   $(X,\mu,T)$ be a joint Furstenberg system associated with the  multiplicative functions $f_1,\ldots, f_\ell\colon \Z\to[-1,1]$ and a sequence of intervals $\bN$.
     Suppose first that  for some $j\in \{1, \ldots, \ell\}$  we have $\E_{n\in\N}|f_j(n)|=0$, say for $j=\ell$. Then all correlations involving the function $f_\ell$ are trivial.  As a consequence, the joint Furstenberg system associated with the functions $f_1,\ldots, f_\ell$ and $\bN$ is isomorphic (in the measure theoretic sense) to the   joint Furstenberg system associated with the functions $f_1,\ldots, f_{\ell-1}$ and $\bN$. Hence, it suffices to  prove Theorem~\ref{T:structure} in the case where $\E_{n\in\N}|f_j(n)|\neq 0$
     for  $j=1,\ldots, \ell$.   As in the first part of the proof of Proposition~\ref{P:constantone} we get
  that $\lE_{p\in \P}(1-|f_j(p)|)=0$
  for $j=1,\ldots, \ell$.  Hence,
  $f_j\sim f_j'$ for some $f_j'\colon \P\to \{-1,1\}$ for $j=1,\ldots, \ell$. As a consequence, in the identity of Theorem~\ref{T:Tao1} we can replace the weights
    $c_{p,m}=\prod_{j=1}^m g_j(p)$ with the weights $c'_{p,m}:=\prod_{j=1}^m g'_j(p)$, where for $j=1,\ldots, m$, if $g_j(p)$ is $f_k(p)$ or $\overline{f_k(p)}$ for some $k\in \{1,\ldots, \ell\}$, then $g_j'(p)$ is
     $f'_k(p)$ or $\overline{f'_k(p)}$ respectively.
       Using  this new identity, we deduce Theorem~\ref{T:structure} as in the case treated above where $f_j(\P)$ is finite for $j=1,\ldots, \ell$.

\section{Proof of Theorems~\ref{T:Sarnak1}-\ref{T:Sarnak2}}\label{S:123}
 We will use the following disjointness result, proved  in  \cite[Corollary~3.13]{FH18}:
 \begin{proposition}
 	\label{P:disjoint}
 	Let $(X,\mu,T)$ be  a system with ergodic components   isomorphic to
 	direct products of    infinite-step nilsystems and  Bernoulli systems. Let   $(Y,\nu,R)$ be a zero entropy system with   at most countably many  ergodic components.
 	\begin{enumerate}
 		\item
 		\label{it:disjoint-1}
 		If the two systems   have disjoint irrational spectrum, then for every  joining  $\sigma$  of the two systems  and   function $F\in L^\infty(\mu)$  that is orthogonal to $\CK(T)$, we have
 		$$
 		\int_{X\times Y} F(x)\, G(y)\, d\sigma(x,y)=0
 		$$
 		for every $G\in L^\infty(\nu)$.
 		
 		\item
 		\label{it:disjoint-2}
 		If the two systems  have no common eigenvalue except $1$,
 then  for every joining $\sigma$ of the two systems and function
 $G\in L^\infty(\nu)$  that is orthogonal in $L^2(\nu)$ to all $R$-invariant functions, we have
$$\int_{X\times Y} F(x)\, G(y)\, d\sigma(x,y)=0
$$
for every $F\in L^\infty(\mu)$.
 	\end{enumerate}
 \end{proposition}

\subsection{Proof of Theorem~\ref{T:Sarnak1}} We follow the argument used in \cite[Section~3.9]{FH18}.
Arguing  by contradiction, suppose that under the assumptions
 of Theorem~\ref{T:Sarnak1}  we do not have uniform convergence to $0$ of the related averages. Then   there exist a strongly aperiodic multiplicative function $f\colon \N\to \mathbb{U}$,  which we extend to $\Z$ in an arbitrary way, a  topological  dynamical system $(Y,R)$,  positive
integers   $N_k\to\infty$,  points $y_k\in Y$, $k\in \N$,
and a function $g_0\in C(Y)$ such that the averages
$$
	\lE_{n\in [N_k]}\, g_0(R^ny_k)\, f(n)
$$
converge to a non-zero number as $k\to\infty$.
After passing to a subsequence  which we denote again by  $([N_k])_{k\in\N}$,  we can further assume that
 the averages
$\lE_{n\in[N_k]}\delta_{R^ny_k}$ converge  (as $k\to\infty$) weak-star to an $R$-invariant  probability measure
$\nu$ and  the limit
\begin{equation}
	\label{E:converge}
	\lim_{k\to\infty} \lE_{n\in[N_k]}\, g(R^ny_k)\prod_{j=1}^m\, f_j(n+h_j)
\end{equation}
exists for all $m\in\N$,  $h_1,\dots,h_m\in\Z$, $f_1,\ldots, f_m\in \{f,\overline{f}\}$, and $g\in C(Y)$.
 Note that, by our assumptions, the system $(Y,\nu,R)$ has  zero entropy and  at most countably many  ergodic components.

Let $X:=\mathbb{U}^\Z$, $T\colon X\to X$ be the shift transformation,  and $x_0\in X$  be defined by
$$
x_0(n):=f(n), \quad n\in\Z.
$$ Then the convergence~\eqref{E:converge} implies  that the limit
$$
\lim_{k\to\infty}\lE_{n\in[N_k]}\,g(R^ny_k)\, \bigl(\prod_{j=1}^m G_{h_j}\bigr)(T^nx_0)
$$
exists for all $m \in \N$, $h_1,\dots,h_m\in\Z$, $g\in C(Y)$, and  $G_h\in \{F_h,\overline{F_h}\}$, $h\in \Z$, where  $F_h(x)=x(h)$, $x\in X$,  $h\in\Z$. Since the algebra generated by  the functions $F_h, \overline{F_h}$, $h\in \Z$,  is dense in $C(X)$  with the uniform topology, we deduce  that the sequence of measures
$$
\lE_{n\in [N_k]}\delta_{(T^nx_0,R^ny_k)}, \quad k\in\N,
$$
converges weak-star to some probability measure $\sigma$ on $X\times Y$ that satisfies
\begin{equation}
	\label{E:sigmaid}
\lim_{k\to\infty}	\lE_{n\in[N_k]}\,g(R^ny_k)\, \prod_{j=1}^m f_j(n+h_j)=\int_{X\times Y} \prod_{j=1}^m G_{h_j}(x)\,g(y)\,d\sigma(x,y)
\end{equation}
for all $m \in \N$, $h_1,\dots,h_m\in\Z$,  $f_1,\ldots, f_m\in \{f,\overline{f}\}$, and $g\in C(Y)$, where $G_h$ is $F_h$ or $\overline{F_h}$ according to whether $f_j$ is $f$ or $\overline{f}$.
By construction, $\sigma$ is invariant under $T\times R$.

The projection of $\sigma$ on $Y$ is the weak-star limit of the sequence of measures $\lE_{n\in [N_k]}\delta_{R^ny_k}$, $k\in\N$, which is the measure  $\nu$,  and thus the corresponding  measure preserving system  has zero entropy and  at most countably many  ergodic components.

The projection of $\sigma$ on $X$ is the weak-star limit of the sequence of measures $\lE_{n\in [N_k]}\delta_{T^nx_0}$, $k\in\N$.
It is thus a $T$-invariant measure  $\mu$ which is the Furstenberg measure associated with $f$ and $\bN=([N_k])_{k\in\N}$ by Proposition~\ref{P:correspondence}. Hence,   $\sigma$ is a joining of the systems  $(X,\mu,T)$ and $(Y,\nu,R)$.

By the $\ell=1$ case of Proposition~\ref{P:factor} and its proof, there exists $d\in \N$ such that $(X,\mu,T)$
is a factor of the system $(X^\Z,\wt \mu_d, S)$,
with factor map  $\pi \colon X^\Z\to X$ given by
$$
(\pi(\ux))(k):=x_k(0), \quad \ux=(x_n)_{n\in\Z}\in X^\Z, \ k\in \Z.
$$
We define the joining $\wt\sigma$  of the systems  $(X^\Z, \wt \mu_d, S)$ and $(Y,\nu,R)$ by
$$
	\int_{X^\Z\times Y} H(\ux)\cdot g(y) \, d\wt\sigma(\ux,y):= \int_{X\times Y} \E_{\wt\mu_d}(H\mid X)(x) \cdot g(y) \, d\sigma(x,y)
$$
for every $H\in L^\infty(\wt \mu_d)$ and $g\in L^\infty(\nu)$, where
$\E_{\wt\mu_d}(H\mid X)$ in $L^1(\nu)$ is determined by the property
$\int_{A}\E_{\wt\mu_d}(H\mid X)\, d\mu=\int_{\pi^{-1}(A)} H\, d\wt \mu_d$ for every $A\in \CX$.

We show now  that the systems $(X^\Z,\wt\mu_d, S)$ and $(Y,\nu,R)$ verify the   assumptions of  part $\text{(i)}$  of Proposition~\ref{P:disjoint}.
By Theorem~\ref{T:Structure-mud},  the system $(X^\Z,\wt\mu_d,S)$  has no irrational spectrum and its
ergodic components are isomorphic to direct products of  infinite-step nilsystems and  Bernoulli systems.
We  show next that the function $ F'_0:=F_0\circ\pi$ is
orthogonal to the rational Kronecker factor of the system $(X^\Z,\wt \mu_d, S)$,
in fact,  we establish the stronger property that  $ F'_0$ is orthogonal to the Kronecker factor of this system.
 By a well known consequence of the spectral theorem for unitary operators,
this  is equivalent to establishing that
\begin{equation}\label{E:Kronecker}
	\E_{n\in\N} \Big|\int_{X^\Z}   F'_0\cdot S^n  \overline{ F'_0}\, d\wt \mu_d\Big|=0.
\end{equation}
By the Definition~\ref{D:tilde} of the measure $\wt\mu_d$   and since for $n\in \N$ we have $ F'_0(\ux) \, \overline{ F'_0}(S^n\ux)=F_0(x_0) \, \overline{F_0}(x_n)$,  we get for every $n\in \N$ that
$$
\int_{X^\Z} F'_0\cdot S^n   \overline{F'_0}\, d\wt \mu_d=\E_{p\in\P_d}\int_X  F_0\cdot T^{pn}  \overline{F_0}\, d \mu.
$$
By \eqref{E:correspondence}, for  every $h\in \N$ we have
$$
\int_X F_0\cdot T^h\overline{F_0}\, d \mu=	\lE_{n\in {\bN} }\,  f(n)\, \overline{f(n+h)}=0
$$
where  the vanishing of the average follows from Theorem~\ref{T:Tao} and our assumption that $f$ is strongly aperiodic. Combining the above identities we get \eqref{E:Kronecker}.

By part $\text{(i)}$  of Proposition~\ref{P:disjoint}, we have
$$
0=\int_{X^\Z\times Y}  F'_0(\ux) \cdot g_0(y)\,d\wt\sigma(\ux,y)= \int_{X\times Y} F_0(x)\cdot  g_0(y)\,d\sigma(x,y)=\lim_{k\to\infty} \lE_{n\in[N_k]}\,g_0(R^ny_k)\, f(n)
$$
where the last identity follows by~\eqref{E:sigmaid}. This contradicts our initial assumption that
$\lim_{k\to\infty}\lE_{n\in [N_k]}\, g_0(R^ny_k)\, f(n)\neq 0$
 and  completes the proof.
\qed

\subsection{Proof of Theorem~\ref{T:Sarnak2}}
We follow the argument used in \cite[Section~3.11]{FH18}.

Arguing  by contradiction, suppose that the conclusion of Theorem~\ref{T:Sarnak2} fails.  Then   there exist
$\ell\in \N$,  multiplicative functions $f_1,\ldots, f_{\ell}\colon \N\to \mathbb{U}$, which we extend to $\Z$
in an arbitrary way,  a  topological dynamical system $(Y,R)$, a point $y_0\in Y$ that is generic for a measure $\nu$ such that the system $(Y,\nu,R)$ has  zero entropy and at most  countably many ergodic components all of which are totally ergodic, and   a function  $g_0\in C(Y)$ orthogonal in $L^2(\nu)$ to all $R$-invariant
functions, such that  for some $h_{0,1},\dots,h_{0,\ell}\in\Z$  the identity \eqref{E:weightedElliott} fails, namely, the averages
\begin{equation}\label{E:000}
	\lE_{n\in [N]} \,g_0(R^ny_0)\, \prod_{j=1}^{\ell}f_j(n+h_{0,j})
\end{equation}
do not converge to $0$ as $N\to\infty$.

Let $X:=(\mathbb{U}^{\ell})^\Z$, $T\colon X\to X$ be the shift transformation,  and $x_0\in X$  be defined by
$$
x_0(n):=(f_1(n),\ldots, f_{\ell}(n)), \quad n\in\Z.
$$
If  $x=(x_1(n), \ldots, x_\ell(n))_{n\in\Z}\in X$, where $x_j(n)\in U$ for $j=1,\ldots, \ell$, $n\in \Z$,  we let
$$
F_{h, j}(x):=x_j(h),\quad  \ h\in \Z,\  j\in \{1,\ldots, \ell\}.
$$
As in the proof of Theorem~\ref{T:Sarnak1}  in the previous subsection,
we define a sequence of intervals $\bN=(N_k)_{k\in\N}$,  with $N_k\to\infty$,  such that the  averages \eqref{E:000}, taken along $\bN$,  converge to some non-zero number,  and  a measure $\sigma$ on $X\times Y$ which is the weak-star limit of the sequence of measures
$$
\lE_{n\in [N_k]}\delta_{(T^nx_0,R^ny_0)}, \quad k\in\N.
$$
In particular, the identity
\begin{equation}
	\label{E:identity0}
	\lE_{n\in\bN}\,g(R^ny_0)\, \prod_{j=1}^\ell f_j(n+h_j)=\int_{X\times Y} \prod_{j=1}^\ell F_{h_j, j}(x)\,g(y)\,d\sigma(x,y)
\end{equation}
holds for all $h_1,\ldots, h_\ell\in \mathbb{Z}$, and $g\in C(Y)$.

By construction, $\sigma$ is invariant under $T\times R$.
By  assumption and the definition of genericity, the projection of $\sigma$ on $Y$ is the measure $\nu$, and thus   the system  $(Y,\nu,R)$ has zero entropy,   at most countably many  ergodic components, and no rational eigenvalue except  $1$. Moreover, the projection of $\sigma$ on $X$ is the weak-star limit of the sequence of measures $\lE_{n\in [N_k]}\delta_{T^nx_0}$, $k\in\N$.
It is thus a $T$-invariant measure    $\mu$ which  is the joint Furstenberg measure associated with the multiplicative functions $f_1,\ldots, f_\ell$  and $\bN$ by Proposition~\ref{P:correspondence}. Hence, by Proposition~\ref{P:factor}, for some $d\in \N$ the system $(X,\mu,T)$
is a factor of the system $(X^\Z,\wt \mu_d, S)$. By Theorem~\ref{T:structure},  the system $(X^\Z,\wt\mu_d,S)$  has no irrational spectrum and its ergodic components are isomorphic to
direct products of   infinite-step nilsystems and  Bernoulli systems.

From the previous discussion  it follows that    the function $g_0$ and the systems $(X^\Z,\wt \mu_d, S)$ and  $(Y,\nu,R)$
 satisfy the hypothesis of the second part    of  Proposition~\ref{P:disjoint}. Hence, for every joining $\wt{\sigma}$ of these
 systems and $\tilde{f}\in  L^\infty(\wt \mu_d)$, we have $\int \tilde{f}(\ux)\, g_0(y)\, d\wt{\sigma}(\ux,y)=0$.
  Since $\sigma$ is a joining of the systems $(X,\mu,T)$ and $(Y,\nu,R)$, and
 the system $(X,\mu,T)$ is a factor of  $(X^\Z,\wt \mu_d, S)$, the measure $\sigma$ can be lifted to a joining $\wt\sigma$ of $(X^\Z,\wt \mu_d,
 S)$ and  $(Y,\nu,R)$. It follows that for every $f\in  L^\infty(\mu)$ we
 have $\int f(x)\, g_0(y)\, d\sigma(x,y)=0$.
We deduce that
$$
 \lE_{n\in \bN} \,g_0(R^ny_0)\, \prod_{j=1}^{\ell}f_j(n+h_{0,j})
 =\int_{X\times Y} \prod_{j=1}^{\ell}F_{h_{0,j},j}(x)\cdot g_0(y)\,d\sigma(x,y)=0.
$$
This contradicts our assumption that $\lE_{n\in \bN} \,g_0(R^ny_0)\, \prod_{j=1}^\ell f_j(n+h_{0,j})\neq 0$ and completes the proof of  Theorem~\ref{T:Sarnak2}.

\subsection{Block complexity and proof of Theorem~\ref{T:superlinear}}
 \label{SS:Superlinear}

We start with some definitions.
Let $A$ be  a non-empty finite set.  The set $A$ is endowed with the discrete topology and $A^\Z$ with the product topology and with the shift $T$. For $n\in\N$, a \emph{word of length $n$}  is a sequence $u=u_1\dots u_n$ of $n$ letters where $u_1,\ldots, u_n\in A$, and we write $[u]=\{x\in A^\Z\colon x_1\dots x_n=u_1\dots u_n\}$.
A \emph{subshift}  is a closed non-empty $T$-invariant subset $X$ of $A^\Z$. It is \emph{transitive} if it has at least one dense orbit under $T$.

Let $(X,T)$ be a transitive subshift that is equal  to the closed orbit of some point $\omega\in A^\Z$.
For every $n\in \N$ we let $L_n(X)$ denote the set of words $u$ of length $n$ such that $[u]\cap X\neq\emptyset$.   Then $L_n(X)$ is also the set of words of length $n$ that occur (as consecutive values) in $\omega$. Note that  the set $L(X):=\bigcup_{n\in \N}L_n(X)$ determines $X$.    The \emph{block complexity} of $X$ or of $\omega$ is defined by $p_X(n)=|L_n(X)|$ for $n\in\N$.
We say that the subshift $(X,T)$ (or the sequence $\omega$) \emph{has linear block growth} if
$\liminf_{n\to\infty} p_X(n)/n<\infty$. We are going to use the following consequence of a result from  \cite{CK} (or \cite[Theorem 7.3.7]{FMo}), that was obtained in  \cite[Section~7.1]{FH18}:
\begin{proposition}
	\label{P:lineargrowth}
	Let $(X,T)$ be a transitive subshift with linear block growth. Then $(X,T)$ admits only finitely many ergodic invariant measures.
\end{proposition}
This result was proved in~\cite{Bos} under the stronger hypothesis that $(X,T)$ is minimal.

\begin{proof}[Proof of Theorem~\ref{T:superlinear}] We argue as in \cite[Section~7.2]{FH18} where a similar result was proved for the Liouville function. Let $A$ be the range of $f$, which  we have assumed to be a finite subset of $\mathbb{U}$.
Suppose  that $f$ has linear block growth. We extend $f$ to a two sided sequence, which we denote by $y_0\in A^\Z$, by letting $y_0(n):=1$ for non-positive $n\in \Z$;  then the extended sequence  still has linear block growth.  Let $Y$ be the closed orbit of $y_0$ in $A^\Z$ and let $R$ be the shift on $Y$. Then $(Y,R)$ is a transitive subshift, and since it has linear block growth  it has zero topological entropy. Moreover,  by Proposition~\ref{P:lineargrowth}   this system admits only finitely many ergodic invariant measures. Note that for every $n\in \N$ we have $f(n)=F_0(R^ny_0)$, where $F_0\colon A^\Z\to\mathbb{U}$ is the map defined by $F_0(y):= y(0)$ for $y=(y(n))_{n\in\Z}\in Y$. By Theorem~\ref{T:Sarnak1} we get
$$
0= 
\lE_{n\in\N} \, \overline{F_0}(R^ny_0)\, f(n) =\lE_{n\in\N}\, |f(n)|^2\neq 0,
$$
where we used our assumption that  $f$ does not converge to zero in logarithmic density.  We have thus established a contradiction and the proof is complete.
\end{proof}

\end{document}